\documentclass{article}

 \usepackage{geometry}
\usepackage{amsmath,amsfonts, amsthm,amssymb,oldgerm}
\usepackage{graphicx}
\usepackage{array}
\usepackage{epsfig}
\usepackage{relsize}
\def\<{\langle}
\def\>{\rangle}
\def\D{\boldsymbol{\mathcal{D}}}

\def\2{L^2}
\def\1{L^2(0,1)}

\def\lam{\lambda}

\newtheorem{thm}{\bf Theorem}[section]

\newtheorem{lem}[thm]{\bf Lemma}

\newtheorem{prop}[thm]{\bf{Proposition}}

\title{Global Well-posedness and Asymptotic Behavior of a Class of Initial-Boundary-Value Problem of the
 Korteweg-de Vries Equation on a Finite Domain}
\author{Ivonne Rivas\\
{\small Department of Mathematical Sciences,} \  \\
{\small University of Cincinnati } \\ {\small Cincinnati, Oh 45221} \\
{\small
email: rivasie@mail.uc.edu} \\
\qquad \vspace{.2in} \\
 Muhammad Usman
\\ {\small  Department of Mathematics} \\ {\small University of Dayton}
\\ {\small
Dayton, Ohio 45469} \\ {\small Muhammad.Usman@notes.udayton.edu}\\
\quad \vspace{.2in}\\ Bing-Yu Zhang \\ {\small Department of
Mathematical Sciences} \\ {\small University of Cincinnati}
\\ {\small Cincinnati,
Ohio 45221} \\ {\small email: zhangb@ucmail.uc.edu}}

\date{}

\begin{document}

\maketitle
\newpage
\begin{abstract}
In this paper, we study a class of initial boundary value problem
(IBVP) of the Korteweg-de Vries equation posed on a finite interval
with nonhomogeneous boundary conditions. The IBVP is known to be
locally well-posed, but its  global $L^2-$ \emph{a priori} estimate
is not available and therefore it is not clear whether its solutions
exist globally or blow up in finite time. It is shown in this paper
that the solutions exist globally as long as their initial value and
the associated boundary data are small, and moreover, those
solutions decay exponentially if their  boundary data decay
exponentially

\end{abstract}

\section{Introduction}
Considered herein  is  an  initial-boundary value problem (IBVP) for
the Korteweg-de Vries equation
 posed on a finite interval $(0,L)$ with nonhomogeneous boundary conditions,
 namely,

    \begin{equation}\label{PP}
        \begin{cases}
        u_t+u_x+u_{xxx}+uu_x=0,\qquad    x\in (0,L),  \ t >0,
        \\ u(x,0)=\phi(x), \\
        u(0,t)=h_1(t),\ u_x(L,x)=h_2(t),\ u_{xx}(L,t)=h_3(t).
        \end{cases}
    \end{equation}

 This IBVP was considered by Colin and Ghidaglia in 2001
 \cite{col-ghida} as a  model for propagation of surface water waves in the situation where
 a wave-maker is putting energy in a finite-length channel from the left $(x=0)$ while the
right end $(x=L)$ of the channel is free (corresponding the case of
$h_2=h_3=0$). In particular, they studied  the IBVP (\ref{PP}) for
its well-posedness in the  space $H^s (0,L)$ and obtained the
following results.

\medskip
\noindent {\bf Theorem A:} \emph{\begin{itemize} \item[(i)] \ Given
$h_j\in C^1([0, \infty)), \ j=1,2,3$ and $\phi \in H^1 (0,L)$
satisfying $h_1(0)=\phi (0)$, there exists a $T>0$ such that the
IBVP (\ref{PP}) admits a solution (in the sense of distribution)
\[ u\in L^{\infty}(0,T; H^1(0,L))\cap C([0,T]; L^2 (0,L)) .\]
\item[(ii)] Assuming $h_1=h_2=h_3\equiv 0$, then for any $\phi \in
L^2 (0,L)$, there exists a $T>0$ such that the IBVP (\ref{PP})
admits a unique weak solution $u\in C([0,T]; L^2 (0,L))\cap L^2
(0,T; H^1 (0,L))$.
\end{itemize}
}

\medskip
The result is temporally local in the sense that the solution $u $
is only guaranteed to exist on the time interval $(0, T)$, where $T$
depends on the size of the initial value $\phi $ and the boundary
data $h_j, j=1,2,3$  in the space $H^1 (0,L)$ (or $L^2 (0,L)$) and
$C^1_b (0, \infty)$, respectively. A problem  arises naturally.

\medskip \emph{{\bf Problem B:} Does the solution exist globally?}

\medskip
Usually, with the local well-posedness in hand, one needs to
establish certain global \emph{a priori estimate} of the solutions
to obtain the global well-posedness. However, this task  turns out
to be  surprisingly  difficult and challenging since the
$L^2-$energy of the solution $u$ of the IBVP (\ref{PP}) is not
conserved as in the situation  of the KdV equation posed on the
whole line $\mathbb{R}$ or on a periodic domain $\mathbb{T}$ even in
the case of homogeneous boundary conditions ($h_j\equiv 0, \
j=1,2,3$). Indeed, for any smooth solution $u $ of the IBVP
(\ref{PP}) with $h_j\equiv 0, \ j=1,2,3$, it holds that
 \[ \frac{d}{dt} \int ^L_0 u^2(x,t)dx =-\frac12 u^2(0,t)+
\frac23 u^3(L,t).\]  The lack of an effective means to deal with the
term $\frac23 u^3(L,t)$ makes it hard to establish the needed
global
 \emph{a priori} estimate for the solutions of  the IBVP (\ref{PP}) in the
space $L^2 (0,L)$. Consequently, Problem B  is open even for the
homogeneous IBVP (\ref{PP}).

\smallskip
In \cite{col-ghida}, Colin and Ghidaglia  provided a partial answer
to Problem B by showing  that \emph{the solution $u$ of the IBVP
(\ref{PP}) exists globally in $H^1(0,L)$} \emph{if the size of its
initial value $\phi \in H^1 (0,L)$ and its boundary values $h_j\in
C^1([0, \infty )), \ j=1,2,3$ are all small.}

\smallskip
Recently, the IBVP (\ref{PP}) has been  studied  by Kramer and Zhang
\cite{kz}, and  Kramer, Rivas and Zhang \cite{rkz} to address an
open question of Colin and Ghidaglia  \cite{col-ghida} regarding
some well-posedness issues of the IBVP (\ref{PP}). They obtained the
following well-posedness results for the IBVP (\ref{PP}) \cite{kz,
rkz}.

\medskip
\noindent {\bf Theorem C}: \emph{Let $s>-1$ and $T>0$  and $r
>0$ be given with
\[ s\ne \frac{2j-1}{2}, \quad j=1,2,3,\cdots . \]
There exists a $T^*
>0$ such that for given $s-$compatible
\footnote{The reader is referred to \cite{kz} for the precise
definition of $s-$compatibility for the IBVP (\ref{PP}). One of the
sufficient conditions for $\phi, h_1,h_2, h_3$ to be $s-$compatible
is $\phi \in H^s_0 (0,L)$ and $$h_1\in H^{\frac{s+1}{3}}_0 (0,T], \
h_2\in H^{\frac{s}{3}}_0 (0,T],\ h_3\in H^{\frac{s-1}{3}}_0
(0,T].$$}
$$\phi \in H^s (0,L), \quad h_1
 \in H^{\frac{s+1}{3}} (0,T), \quad  h_2\in H^{\frac{s}{3}}(0,T),
 \quad h_3\in H^{\frac{s-1}{3}}(0,T) $$  satisfying
 \[ \|\phi \|_{H^s(0,L)} +\| h_1\|_{H^{\frac{s+1}{3}}(0,T)} +\|h_2\|_{H^{\frac{s}{3}}(0,T)}
 +\|h_3\|_{H^{\frac{s-1}{3}}(0,T)} \leq r,\]
 the IBVP (\ref{PP}) admits a unique solution
 \[ u\in C([0,T^*]; H^s(0,L))\cap L^2 (0,T^*; H^{s+1} (0,L)) .\]
 Moreover, the solution $u$ depends Lipschitz continuously on $\phi $ and
 $h_j, j=1,2,3$ in the corresponding spaces.
}

\medskip
{\bf Remarks:}

\begin{itemize}
\item[(1)] The well-posedness presented in Theorem C is in its full
strength; it includes uniqueness, existence and (Lipschitz)
continuous dependence as well as persistence (the solution $u$ forms
a continuous flow in the space $H^s (0,L)$).

\item[(2)]For the well-posedness of the IBVP (\ref{PP}) in the space
$H^s (0,L)$, the regularity conditions imposed on the boundary data
$h_j, \ j=1,2,3$ are optimal. In particular, when $s=1$, it is only
required that  $h_1\in H^{\frac23} (0,T)$, $h_2 \in H^{\frac13}
(0,T)$ and $h_3\in L^2(0,T)$ instead of $h_1, h_2, h_3 \in C^1_b
(0,T)$ as in Theorem A.
\end{itemize}

Nevertheless, the well-posedness result presented in Theorem C is
still temporal local.   The question whether the solution exists
globally remains open.  In this paper, we continue to study the IBVP
(\ref{PP}) but emphasizing on the issues of its global
well-posedness in the space $H^s(0,L)$ and the long time asymptotic
behavior of those globally existed solutions. In order to describe
our results more precisely, we first introduce some notations.

\smallskip
For given $s\geq 0$, $t\geq 0$ and $T>0$, let $\vec{h}:= (h_1,h_2,
h_3)$,
\[ B^s_{(t,t+T)}= H^{\frac{s+1}{3}}(t, t+T)\times
H^{\frac{s}{3}}(t, t+T)\times H^{\frac{s-1}{3}}(t, t+T)\] and
\[
Y^s_{(t,t+T)} =C([t, t+T];H^s(0,L))\cap L^2 (t, t+T; H^{s+1} (0,L)),
\quad X^s_{(t, t+T)}=H^s(0,L)\times B^s_{(t, t+T)} \] In addition,
let
\[ B^s_T =\{ \vec{h}\in B_{(t,t+T)}^s \ for \ any \ t\geq 0: \
\sup_{t\geq 0} \| \vec{h}\|_{B^s_{(t, t+T)}} < \infty \}, \quad
X^s_T= H^s(0,L)\times B^s_T,
\] and
\[ Y^s_T =\{ u\in Y_{(t,t+T)}^s \ for \ any \ t\geq
0: \ \sup_{t\geq 0} \| u\|_{Y^s_{(t, t+T)}} < \infty \} \] Both
$B^s_T$ and $Y^s_T$ are Banach spaces equipped with the norms
\[  \| \vec{h}\|_{B^s_T} :=\sup _{0< t< \infty} \| \vec{h}\|_{B^s_{(t, t+T)}} \]
and
\[  \| u\|_{Y^s_T} :=\sup _{0< t< \infty} \| u\|_{Y^s_{(t, t+T)}}
,\]respectively. If $s = 0$, the superscript $s$ will be omitted
altogether, so that
\[ B_{(t, t+T)}= B^0_{(t,t+T)}, \ X_{(t, t+T)}= X^0_{(t, t+T)}, \
Y_{(t, t+T)}= Y^0_{(t, t+T)} \] and
\[ B_T=B^0_T, \ X_T=X^0_T, \ Y_T=Y^0_T .\]
Moreover, because of their frequent occurrence, it is convenient to
abbreviate the norms of $u$ and $h$ in the space $H^s (0,L)$ and
$H^s (a,b))$, respectively, as
\[ \|u\| _s= \|u\|_{H^s (0,L)}, \qquad |h|_{s,(a,b)} ,\]
and
\[ \| u\| =\|u\|_{L^2 (0,L)}, \qquad |h|_{(a,b)}= \|h\|_{L^2 (a,b)} .\]

\smallskip
 The main results of this paper are summarized in the following two
 theorems. The first one states that the small amplitude solutions
 exist globally.
 \begin{thm}[Global well-posedness]\label{global} Let $s\geq 0$ with
\[ s\ne \frac{2j-1}{2}, \quad j=1,2,3,\cdots . \]
 There exist  positive constants $\delta $ and
 $T$ such that
 for any $s-$compatible $(\phi , \vec{h})\in X^s_T$ with
 \[ \| (\phi , \vec{h})\|_{X^s_T} \leq \delta , \]
 the IBVP (\ref{PP}) admits a unique solution $u\in Y^s_T$.
 \end{thm}
 The second one states that the small amplitude solutions decay
 exponentially as long as their boundary data decay exponentially.
 \begin{thm}[Asymptotic behavior]\label{long time} If, in addition to the assumptions of Theorem
 \ref{global}, there exist   $\gamma _1>0$,  $C_1>0$   and $g\in B_{T}^s$ such that
 \[ \| \vec{h}\|_{B^s_{(t, t+T)}} \leq g(t)e^{-\gamma _1 t } \quad for \
 t\geq 0,\] then there exists $\gamma  $ with $0< \gamma  \leq
 \gamma _1 $  and $C_2 >0$  such that the corresponding solution $u$
 of the IBVP(\ref{PP}) satisfies
 \[ \|u\|_{Y^s_{(t, t+T)}} \leq C_2\| (\phi , \vec{h})\|_{X^s_T} e^{-\gamma t} \quad for \
 t\geq 0. \]
 \end{thm}

The study of the initial-boundary-value problems of the KdV equation
posed on the finite domain started as early as in 1979  by Bubnov
\cite{bub-1} and has been intensively studied in the past twenty
years for its well-posedness following the advances of the study of
the pure initial value problems of the KdV equation posed either on
the whole line $\mathbb{R}$ or on a torus $\mathbb{T}$. The
interested readers are referred to \cite{bub-1, bub-2, bsz-finite-1,
bsz-finite-2, col-ghida, Fam01, Fam07, holmer, kz, rkz} and the
references therein for an overall review for the well-posedness of
the IBVP of the KdV equation posed a finite domain and
\cite{bsz-half-1, bsz-half-2, ck, Fam83, Fam89, Fam99, Fam04,
holmer} for the IBVP of the KdV equation posed on the half line
$\mathbb{R}^+$.

 The paper is organized as follows.

 \smallskip
 ---- In section 2, we consider  the associated linear problem
 \begin{equation}\label{LP}
        \begin{cases}
        u_t+(au)_x+u_{xxx}=0,\qquad    x\in (0,L),  \ t >0,
        \\ u(x,0)=\phi(x), \\
        u(0,t)=h_1(t),\ u_x(L,x)=h_2(t),\ u_{xx}(L,t)=h_3(t).
        \end{cases}
 \end{equation}
 where $a=a(x,t)$ is a given function. Attention will be first turned to the situation that $a\equiv 0$ and  all boundary
 data $h_1, \ h_2$ and $h_3$ are zero:
 \begin{equation}\label{HLP}
        \begin{cases}
        u_t+u_x+u_{xxx}=0,\qquad    x\in (0,L),  \ t >0,
        \\ u(x,0)=\phi(x), \\
        u(0,t)=0,\ u_x(L,x)=0,\ u_{xx}(L,t)=0.
        \end{cases}
 \end{equation}
 It will be shown that the linear  KdV equation in (\ref{HLP}) behaves like
 a heat equation;
 \begin{itemize}
 \item[(i)] it possesses a remarkably strong smoothing property: for any
 $\phi \in L^2(0,L)$ , the corresponding solution $u(t)$ belongs to the
 space $H^{\infty} (0,L)$ for any $t>0$.
 \item[(ii)] its solution $u$ decays exponentially in the space $H^s
 (0,L)$ (for any given $s\geq 0$) as $t\to \infty$.
 \end{itemize}
These heat equation like properties of the IBVP (\ref{HLP}) enable
us to show that for any $s\geq 0$ there exists a $T>0$ such that if
$a\in X^s_T$ and  $\|a\|_{X_T^s} $ is small enough, then for any
$s-$compatible $(\phi , \vec{h})\in X^s_T$ with $G(t)=\|
\vec{h}\|_{B^s_{(t, t+T)}}$ decays exponentially,  the corresponding
solution $u$ of (\ref{LP}) also decays exponentially in the space
$H^s (0,L)$ as $t\to \infty $.

\medskip
---- In Section 3, the nonlinear IBVP (\ref{PP}) will be the focus
of our attention.  The proofs will be provided for both Theorem 1.1
and Theorem 1.2. As one can see from the proofs, the results
presented in Theorem 1.1 and Theorem 1.2 for the nonlinear IBVP
(\ref{PP}) are more or less small perturbation of the results
presented in Section 2 for the linear IBVP (\ref{LP}) and therefore
are  essentially linear results..

\medskip
---- The paper is ended with concluding remarks  given in Section
4. A comparison will be made between the IBVP (\ref{PP}) and the
following IBVP of the KdV equation posed on $(0,L)$:

\begin{equation}\label{others}
        \begin{cases}
        u_t+u_x+ uu_x+ u_{xxx}=0,\qquad    x\in (0,L),  \ t >0,
        \\ u(x,0)=\phi(x), \\
        u(0,t)=0,\ u(L,x)=0,\ u_{x}(L,t)=0.
        \end{cases}
 \end{equation}
We will see that, although there is just a slight difference between
the boundary conditions of the IBVP (\ref{PP}) and  the IBVP
(\ref{others}), there is a big difference between the global
well-posedness results for the IBVP (\ref{PP}) and the IBVP
(\ref{others}). While   only small amplitude solutions the IBVP
(\ref{PP}) exist globally, all solutions of the IBVP (\ref{others}),
large or small, exist globally instead.

\medskip
In addition,  the IBVP (\ref{PP}) will also be shown in this
section, to possess a time periodic solution $u^*$ if the boundary
forcing $\vec{h}$ is time periodic with small amplitude. Moreover,
this time periodic solution $u^*$ is locally exponentially stable.
 \section{Linear problems}
 \setcounter{equation}{0}

  In this section, consideration is first directed to the IBVP of linear KdV equation
  with
  homogeneous boundary conditions
  \begin{equation}\label{2.1}
  \begin{cases} u_t +u_x +u_{xxx}=0, \quad x\in (0,L),\  t>0,\\
  u(x,0)=\phi (x), \\ u(0,t)=u_x (L,t)=u_{xx} (L,t)=0 .\end{cases}
  \end{equation}
Its solution u can be written in the form
\[ u(x,t) = W(t) \phi \]
 where $W(t)$ is the $C^0$-semigroup in the space $L^2 (0,L)$ generated by the operator
   $$Ag:=-g'''-g'$$ with domain
  $$\D(A)=\{g\in H^3(0,L):  g(0)=g'(L)=g''(L)=0\}.$$
  The following estimate  can be found in \cite{kz}.
  \begin{prop}\label{pro2.1} Let $T>0$  be given. There exists a constant $C>0$
  depending only on $T$ such that for any $\phi \in L^2 (0,L)$,
  \[ \|u\|_{Y_{(0,T)}} \leq C_T \|\phi \|  .\]
  \end{prop}
  Our main concern is its long time asymptotic behavior.  As it
  holds that, for any smooth solution $u$ of the IBVP (\ref{2.1}),
  \[\frac{d}{dt}\int ^L_0 u^2(x,t) dx = -u^2(L,t)-u^2_x(0,t), \]
  one  may wonder if its $L^2-$energy decays as $t\to \infty$.
  A detailed spectral analysis of the operator  $A$ is needed  for the investigation.

\medskip
 Note that both $A$  and its adjoint operator $A^*$ are
\emph{dissipative} operator.  Indeed,  the adjoint operator of $A$
is given by
  \[ A^* f=f' +f''' \]
  with the domain
  $$\D(A ^*)=\{f\in H^3(0,L):  f(0)=f'(0)=0, f(L)+ f''(L)=0\}.$$
  A direct calculation shows that
  \[ \int ^L_0 g(x) (Ag)(x)dx =-\frac12 (g'(0))^2, \qquad \int ^L_0
  f(x)(A^*f)(x) dx =-\frac12 (f'(L))^2 \]
  for any $g\in \D(A)$ and $f\in \D (A^*)$. Thus both $A$ and $A^*$
  are dissipative operators.

\begin{lem}\label{noimag} The spectrum  $\sigma (A)$ of $A$ consists
of all eigenvalues  $\{ \lambda _k\}^{\infty}_{1} $ with
\[ Re \, \lambda _k <0, \qquad k=1,2,3,\cdots \]
and
\begin{equation}\label{expansion}
\lambda_k = - \frac{8 \pi^3 k^3}{3 \sqrt{3}L^3}+ O(k^2) \quad \quad
\text{as} \quad k\to \infty
\end{equation}
\end{lem}
\begin{proof} Since $A$ is dissipative and  the resolvent operator
 $R(\lambda, A)= (\lambda I-A)^{-1}$ ($\lambda \in \rho (A)$) is compact,   $$\sigma (A)= \sigma _p (A)=\{ \lambda
 _k\}^{\infty}_1$$
 with $\lambda _k \to \infty $ as $ k\to \infty $ and
 \[ Re\, \lambda _k \leq 0, \quad k=1,2,3, \cdots .\]
 We thus need to show that
\[ Re\, \lambda _k \ne  0, \quad k=1,2,3, \cdots .\]
 Suppose $i\mu\in \sigma_p(A)$ for some real  number $\mu$. There
exists $0\ne f\in \D(A)$ such that
\begin{equation}\label{auxprob}
    \begin{cases}
     -f'''(x)-f'(x)=i\mu f(x) ,\quad   x\in(0,L)\\
     f(0)=f'(L)=f''(L)=0.
     \end{cases}
\end{equation}
Multiply both sides of the equation in (\ref{auxprob}) by $\bar{f}$
and integrate over $(0,L)$. Integration by parts leads to
\[i \mu\int ^L_0 |f(x)|^2 dx = -\frac12 |f'(0)|^2 -\frac12 |f(L)|^2
.\] Consequently,  either
\[ \mu =0, \quad f'(0)=f(L)=0 \]
or
\[ \int ^L_0 |f(x)|^2 dx=0, \quad f'(0)=f(L)=0.\]
Recall that $ f(0)=f'(L)=f''(L)=0$.  Therefore, in either cases, we
have $f\equiv 0$. That is a contradiction.

\medskip
 To show that  (\ref{expansion}), we may assume that $Af=-f'''$. The
 case of $Af=-f'''-f'$ follows from standard perturbation theory
 (cf. \cite{kato}).

 \smallskip
  Assuming that  $Re\,\lambda <0$. By symmetry, we only need to consider the case  that $Im \,\lambda \leq 0$.
  Denote the three cube roots of $-\lambda $ by $\mu _1, \ \mu_2 ,
  \mu _3$. These must have distinct real parts; let $\mu _1$ be the
  unique root such that $0\leq arg (\mu _1)\leq \pi /6$ and
  \[ \mu _2 =e^{\frac{2\pi i}{3}}\mu _1 :=\rho \mu _1, \qquad \mu _3
  =\rho ^2 \mu _1 .\]
  The  solution of
\begin{equation}\label{phi}
    \begin{cases}
    \lambda \phi(x) + \phi'''(x)=0, \quad  x\in (0,L),\\  \\
    \phi(L)=\phi'(L)=\phi''(L) =0
    \end{cases}
\end{equation}
is then given by
 $$\phi(x)= C_1 e^{\mu_1 x} + C_2 e^{\mu_2 x}+C_3 e^{\mu_3 x}$$ with  $C_1,C_2$ and $C_3$ satisfying

\begin{eqnarray*}
    C_1+C_2+C_3&=& 0,\\
     \mu_1 e^{\mu_1 L} C_1 +  \mu_2 e^{\mu_2 L} C_2 +  \mu_3 e^{\mu_3 L}C_3&=&0,\\
      \mu_1^2 e^{\mu_1L} C_1 +  \mu_2^2 e^{\mu_2L} C_2 +  \mu_3^2
      e^{\mu_3L}C_3 &=&0.
\end{eqnarray*}
Setting the determinant of the coefficient matrix equal to zero,
\[ e^{(\mu_2+\mu_3)L}\mu_2 \mu_3 (\mu_3-\mu_2)
 + e^{(\mu_1 +\mu _2)L} \mu_1  \mu_2 (\mu_2-\mu_1)+
 e^{(\mu_3+\mu_1)L
 }(\mu_1-\mu_3)\mu_3\mu_1=0 .\]
 By the assumptions, $Re\,\mu_1 > 0$,  $Re\,\mu_2 <0$ and $Re\,\mu_3 \leq
 0$. Furthermore,
 $$Re\,\mu _1) \to +\infty \quad and \quad Re\,\mu _2 \to -\infty \quad  as \quad \lambda
 \to \infty .$$
Neglecting the term $ e^{(\mu_2+\mu_3)L}\mu_2 \mu_3 (\mu_3-\mu_2)$,
which is very small for large $\lambda $, we arrive at the equation
  $$e^{(\mu _2 -\mu _3 )L}=\rho +\rho ^2,$$
  or
  \[ e^{i\sqrt{3}\mu _1L}=-1 .\]
  Therefore,
\begin{eqnarray*}
\mu_{1,k} \sim \frac{(1+2k)\pi}{\sqrt{3}L}.
\end{eqnarray*}
As $\lambda_k +\mu^3=0,$
\begin{gather*}
\lambda_k=-\mu _{1,k}^3=  - \left(  \frac{8 \pi^3
k^3}{3\sqrt{3}L^3}+ O(k^2)\right) \quad \text{as} \ k\to \infty.
\end{gather*}
\end{proof}

Next lemma gives an asymptotic estimate of the resolvent operator
$R(\lambda , A)$ on the pure imaginary axis.
\begin{lem}\label{resolvent}
$$\| R(iw,A)\| = O(|w|^{-2/3}), \quad  \text{as} \quad |w|\to \infty.$$
\end{lem}
\begin{proof}
Letting  $\lambda \in \rho(A)$ and $ f\in \2(0,L)$ and   defining
$w=(\lambda I - A)^{-1}f$. In other words,  $w$ satisfies
\begin{equation}\label{wprob}
    \begin{cases}
     \lambda w + w'''+w'= f,\\
     w(0)=w'(L)=w''(L)=0.
    \end{cases}
\end{equation}
Its  solution is given by
$$w(y,\lambda)= \int_0^L G(y,\xi; \lambda) f(\xi) d\xi$$
where $G(y,\xi; \lambda)$,  the Green function of (\ref{wprob}),
solves
\begin{equation}\label{green}
    \begin{cases}
    G'''(y,\xi; \lambda) + G'(y,\xi;\lambda) + \lambda G(y,\xi;\lambda)=\delta (y-\xi)\\
    G(0,\xi;\lambda)=G'(L,\xi;\lambda)=G''(L,\xi,;\lambda)=0,
    \end{cases}
\end{equation}
the prime notation  representing $\frac{d}{dy}$.  With $s_j, \
j=1,2,3,$ being the solutions of \[ s^3+s+\lambda =0,\]
$G(y,\xi;\lambda)$ has the form
$$G(y,\xi;\lambda)= \sum_{j=1}^3 c_j e^{s_j (y-\xi)} + H(y-\xi) \left(\sum_{j=1}^3
\hat{c}_j e^{s_j (y-\xi)} \right)$$
with
\begin{equation*}
    H(x)= \begin{cases}  1& \ if \ x>0,\\
    0&  \ if \ x\le 0, \end{cases}
\end{equation*}
the coefficients $\hat{c}_1, \ \hat{c}_2$ and $\hat{c}_3$ satisfying
\begin{eqnarray*}
\hat{c}_1 + \hat{c}_2 +\hat{c}_3 &=& 0,\\
\hat{c}_1 s_1 + \hat{c}_2 s_2 +\hat{c}_3 s_3&=& 0,\\
\hat{c}_1 s_1^2 + \hat{c}_2 s_2^2+\hat{c}_3  s_3^2&=& 1\\
\end{eqnarray*}

and the coefficients  $c_1, \ c_2, \ c_3 $ solving
   \begin{equation} \label{system}\left(
        \begin{array}{ccc}
    1 & 1 & 1 \\
    s_1 e^{s_1L} & s_2 e^{s_2L} &  s_3 e^{s_3L}\\
    s_1^2 e^{s_1L} & s_2^2 e^{s_2L} &  s_3^2 e^{s_3L}\\
         \end{array} \right)
         \left(
         \begin{array}{c}
         c_1 e^{-s_1\xi}\\
         c_2 e^{-s_2 \xi}\\
         c_3 e^{-s_3 \xi}
         \end{array}\right)=
         \left(
         \begin{array}{c}
         0\\
         d_2\\
         d_3
         \end{array}
         \right)
    \end{equation}
where
\begin{gather}\label{d2d3}
    d_2=-\left( \hat{c}_1s_1 e^{s_1(L-\xi)} + \hat{c}_2s_2 e^{s_2(L-\xi)}+\hat{c}_3s_3 e^{s_3(L-\xi)}\right),\\
    d_3=-\left( \hat{c}_1s_1^2 e^{s_1(L-\xi)} + \hat{c}_2s_2^2 e^{s_2(L-\xi)}+\hat{c}_3s_3^2
    e^{s_3(L-\xi)}\right).
\end{gather}
A direct computation  shows  that
$$\hat{c}_1=\frac{1}{(s_2-s_1)(s_3-s_1)}, \quad
\hat{c_2}=\frac{1}{(s_1-s_2)(s_3-s_2)} \quad  \text{and} \quad
\hat{c_3}=\frac{1}{(s_1-s_3)(s_2-s_3)}$$ and
\[ c_1 = e^{s_1\xi}\frac{\Delta _1}{\Delta},\quad c_2 = e^{s_2\xi}\frac{\Delta
_2}{\Delta},\quad c_3 = e^{s_3\xi}\frac{\Delta _3}{\Delta},\] where
$\Delta$ is the determinant of the coefficients matrix  $A$ of
system (\ref{system}),
$$\Delta = e^{s_2 +s_3} s_2 s_3 (s_3-s_2) + e^{s_1 +s_3}s_3 s_1 (s_1-s_3) + e^{s_1+s_2}s_1s_2(s_2-s_1,)$$
and $\Delta_j$ be the determinant of the matrix $A$  with the
$j-$column  replacing by the vector $(0,d_2,d_3)^T$
 for $j=1,2,3$.
Recall that  $ s_j$ ( $j=1,2,3$) is the solution of
$$s^3+s+\lambda=0.$$ If we let  $\lambda=i\omega$, then $ s_j=i \mu _j$ with $\mu _j$
solves
$$\mu^3-\mu -\omega=0$$ and therefore,
\begin{equation*}
 \mu_1=\alpha +\beta,\quad
 \mu_2=\rho ^2 \alpha +\rho \beta,\quad  \mu_3=\rho \alpha +\rho ^2\beta,
\end{equation*}
with $\rho =e^{\frac23 \pi}$,
\[ \alpha =\left (\frac{\omega}{2}+\sqrt{d}\right )^{\frac13}, \qquad \beta =\left (\frac{\omega}{2}-\sqrt{d}\right
)^{\frac13}\] where
\[ d=\frac{\omega ^2}{4}-\frac{1}{27} .\]

Thus, as $\omega ^{\frac13}:=b \to \infty $,
\begin{eqnarray*}
\mu_1&=& b + \frac13 b^{-1} +O( b^{-5}),\\
\mu_2&=&(-\frac12+\frac{\sqrt{3}}{2}i) b -  \frac13(\frac12+\frac{\sqrt{3}}{2}i) b^{-1}+O( b^{-5}),\\
\mu_3&=&(-\frac12-\frac{\sqrt{3}}{2}i) b -  \frac13 (\frac12-\frac{\sqrt{3}}{2}i)b^{-1}+O( b^{-5}).\\
\end{eqnarray*}
As a result, as $b\to \infty $,
\begin{eqnarray*}
s_1&=& i \mu_1=  ib + O(b^{-1}),\\
s_2&=& i \mu_2=  (-\frac12i-\frac{\sqrt{3}}{2})b + O(b^{-1}),\\
s_3&=& i \mu_2=  (-\frac12i+\frac{\sqrt{3}}{2})b + O(b^{-1}),
\end{eqnarray*}
and, asymptotically, as $b\to \infty $,
\begin{gather*}
d_2\sim -s_3  \hat{c}_3 e^{s_3 (L- \xi)}, \\
d_3\sim-s_3^2  \hat{c}_3e ^{s_3(L-  \xi)}.
\end{gather*}
It follows similarly, as $b\to \infty $,
\begin{eqnarray*}
\Delta &\sim&  s_3 s_1 (s_1-s_3) e^{(s_1+s_3)L},\\
\Delta_1
  &\sim & -s_3  \hat{c}_3 e^{-s_3 (L- \xi)}
\left| \begin{array}{ccc}
 0 & 1 & 1\\
1 &0  & 0\\
s_3 & 0 & 0
 \end{array} \right| \\
    &= 0,\\
\Delta_2&\sim & \left| \begin{array}{ccc}
    1& 0 & 1\\
    s_1 e^{s_1L} & -s_3  \hat{c}_3 e^{s_3 (L- \xi)} & s_3 e^{s_3L}\\
    s_1^2e^{s_1L}&-s_3^2  \hat{c}_3e ^{s_3(L-  \xi)}&s_3^2 e^{s_3L}
\end{array} \right|\\
&\sim& -\hat{c}_3 s_3 e^{s_3(L-\xi)} \left|
\begin{array}{ccc}
1 & 0 &1\\
0& 1 & 0\\
e^{s_1L}s_1(s_1-s_3) & s_3& 0
\end{array}
\right|\\
&\sim& \hat{c}_3 s_3 s_1 (s_1-s_3)e^{s_3(L-\xi)+s_1},\\
\Delta_3 &\sim & \left| \begin{array}{ccc}
    1& 1 & 0\\
    s_1 e^{s_1L} & s_2 e^{s_2L} &  -s_3  \hat{c}_3 e^{s_3 (L- \xi)} \\
    s_1^2e^{s_1L}&s_2^2 e^{s_2L}&-s_3^2  \hat{c}_3e ^{s_3(L-  \xi)}
\end{array} \right|\\
&\sim& -\hat{c}_3 s_3 e^{s_3(L-\xi)} \left|
\begin{array}{ccc}
1& 1 & 0\\
0&0&1\\
s_1 e^{s_1L}(s_1-s_3) &  0& s_3
\end{array}
\right|\\
&\sim& -\hat{c}_3 s_3 e^{s_3(L-\xi)} s_1 e^{s_1L}(s_1-s_3).
\end{eqnarray*}

Hence,  as $b\to \infty $,
\begin{equation*}
    c_1\sim 0, \quad \quad c_2\sim \hat{c}_3 e^{(s_2 - s_3) \xi}, \quad  \quad  c_3\sim -\hat{c}_3.
\end{equation*}

Plugging these coefficients in the definition of the Green's
function for IBVP (\ref{green}) and considering the case where $y\le
\xi$,
\begin{eqnarray*}
G(y,\xi; i\omega) &=& c_1 e^{s_1(y-\xi)}+c_2 e^{s_2(y-\xi)}+c_3 e^{s_3(y-\xi)}\\
&=& \hat{c}_3 e^{-s_3\xi +s_2 y}-\hat{c}_3 e^{s_3(y-\xi)}\\
&\sim& \hat{c}_3 e^{\frac{i}2 b(\xi-y) - \frac{\sqrt{3}}2 b(\xi+y) }
-  \hat{c}_3 e^{(y-\xi)(\frac{-i}{2}+ \frac{\sqrt{3}}2)b}.
\end{eqnarray*}
Therefore, since $\hat{c}_1,\hat{c}_2,\hat{c}_3 \sim b^{-2}$

\begin{eqnarray*}
\left| G(y,\xi;ib^3) \right| &\le& \hat{c}_3 M_N\\
&\le& b^{-2}M_N\\
&\le& w^{-2/3} M_N
\end{eqnarray*}

for  $\{ e^{-s_3 \xi + s_2 y}, e^{s_3(y-\xi)}\} \le M_N$ for $b>N$.  Now, considering the cases where $ y>\xi$,

\begin{eqnarray*}
G(y,\xi;i\omega) &=& c_1 e^{s_1(y-\xi)} +c_2 e^{s_2(y-\xi)}+c_3
e^{s_3(y-\xi)} + \hat{c}_1 e^{s_1(y-\xi)}
 +\hat{c}_2 e^{s_2(y-\xi)}+\hat{c}_3 e^{s_3(y-\xi)}\\
&\approx & \hat{c}_3 e^{-s_3 \xi + s_2 y}+ \hat{c}_1 e^{x_1(y-\xi)} + \hat{c}_2 e^{s_2(y-\xi)}\\
&\approx &\hat{c}_1 e^{s_1(y-\xi)} +\hat{c}_2 e^{s_2(y-\xi)}+ \hat{c}_3 e^{-\frac{\sqrt{3}}2 b (\xi+y)} \\
&\approx & M_N b^{-2}.
\end{eqnarray*}

Since $\{  e^{s_1(y-\xi)} ,   e^{s_2(y-\xi)}, e^{-\frac{\sqrt{3}}2 b (\xi+y)}\} \le M_N$ as $b>N$,  so

$$| G(y, \xi, ib^3)| \le M_N b^{-2} \quad \text{as}  \quad b>N,$$

and we can conclude in general that $\forall (y,\xi)$
 $$| G(y, \xi, ib^3)| \le M_N b^{-2} \quad \text{as}  \quad b>N.$$
Notice that if we select $\lam = -i \omega$, the computations are
similar and we get the same asymptotic behavior
 for the Green's function (\ref{green}).\\
\end{proof}

The following estimate then follows from Lemma \ref{noimag}, Lemma
\ref{resolvent} and \cite{Pazy}.
\begin{prop}\label{decay-1}  There exists a $\gamma >0$ such that  for any $\phi
\in L^2 (0,L)$, $$u(t)=W(t)\phi \in H^{\infty} (0,L) \quad  for \
any \ t >0$$ and
\[ \| u(t)\|  \leq Ce^{-\gamma t} \| \phi \|
\quad for \ all \ t\geq 0 .\]
\end{prop}

Combining Propositions \ref{pro2.1} and \ref{decay-1} gives us
\begin{thm}\label{decay-2}
There exists a $\gamma >0$ such that for  $T>0 $ there exists a
constant $C=C_T >0$ such that \[ \| u\|_{Y_{(t, t+T)}} \leq C_T
\|\phi \| e^{-\gamma t} \ for \ all \ t\geq 0.\]
\end{thm}

Now we turn to consider the IBVP of the KdV equation with the
nonhomogeneous boundary conditions.

  \begin{equation}\label{HP}
    \begin{cases}
    u_t+u_x+u_{xxx}=0, \qquad x\in (0,L),\\
    u(x,0)=\phi(x), \\
     u(0,t)=h_1, \quad u_x(L,t)=h_2,  \quad u_{xx}(L,t)=h_3.
    \end{cases}
\end{equation}
Its solution can be written as \[ u(x,t)= W(t) \phi + W_b (t)\vec{h}
\]
where $\vec{h} = (h_1,h_2, h_3)$ and $W_b (t)$ is the boundary
integral operator associated to the IBVP (\ref{HP}) whose explicit
representation formula can be found in \cite{rkz}. The following
estimate is  from \cite{kz,rkz}.
\begin{prop}\label{pro2.9} Let $T>0$ be given. There exists a constant $C=C_T$
depending only on $T$ such that for any $\vec{h}\in B_{0,T}$ and
$\phi \in L^2 (0,L)$, then the IBVP (\ref{HP}) admits a unique
solution $u\in Y_{(0,T)}$ and, moreover,
\begin{equation}\label{2.2}
\|u\|_{Y_{(0,T)}} \leq C_T (\|\phi \| + \| \vec{h}\|_{B_{(0,T)}})
.\end{equation}
\end{prop}
Note that the estimate (\ref{2.2}) can be written as
\[ \|u\|_{Y_{(t, t+T)}}\leq C_T \|u(\cdot, t)\|  +C_T
\|\vec{h}\|_{B_{(t,t+T)}} \ for \ any \ t\geq 0\] because of the
semigroup property of the IBVP (\ref{HP}).

\medskip
Attention now is turned to the IVP of a linearized KdV--equation
with a variable coefficient $a=a(x,t)$, namely,

\begin{equation}\label{SP}
    \begin{cases}
    u_t +(au)_x+u_x +u_{xxx}=0, &  x\in(0,L),\\
    u(x,0)=\phi(x),\\
    u(0,t)=h_1 (t), \ u_x(L,t)=h_2(t), \ u_{xx}(L,t)=h_3 (t).
    \end{cases}
\end{equation}
The following result is known \cite{kz}.
\begin{prop}\label{pro2.7} Let $T>0$ be given. Assume that $a\in Y_{(0,T)}$.
Then for any $\phi \in L^2 (0,L), \vec{h} \in B_{(0,T)}$, the IBVP
(\ref{SP}) admits a unique solution $u\in Y_{(0,T)}$ satisfying
\[ \| u\|_{Y_{(0,T)}} \leq \mu (\| a\|_{Y_{(0,T)}})\left ( \| \phi
\| +\|\vec{h}\|_{B_{(0,T)}}\right ) ,
 \] where $\mu : \mathbb{R}^+ \to \mathbb{R}^+$ is a $T-$dependent
continuous nondecreasing function independent of $\phi $ and
$\vec{h}$.
\end{prop}
The next theorem presents an asymptotic estimate for solutions of
the IBVP (\ref{SP}), which will play  an important role in studying
asymptotic behavior of the nonlinear IBVP in the next section.
\begin{thm}\label{th2.8} There exists a $T>0$, $r>0$ and $\delta >0$ such
that if $a\in Y_T$ with
\[ \| a\|_{Y_T} \leq \delta ,\]
then for any $\phi \in L^2 (0,L), \vec{h} \in B_{T}$, the
corresponding solution $u$ of (\ref{SP}) satisfies
\[ \| u
\|_{Y_{(t, t+T)}}\leq C_1 e^{-rt}\|\phi \|   + C_2  \|
\vec{h}\|_{B_T}\] for any $t\geq 0$ where $C_1 $ and $C_2$ are
constants independent of $\phi $ and $\vec{h}$. Furthermore, if
\begin{equation}\label{ccc} \|\vec{h}\|_{B_{(t, t+T)}}\leq g(t) e^{-\nu t} \ for \ all \
t\geq 0 \end{equation} with $\nu >0$,  $g\in B_T$ and $\| g\|_{B_T}
\leq \delta _2$, then there exist  $0< \gamma \leq \max \{ r, \nu \}
$ and $C>0$ such that
\[  \|u\|_{Y_{(t, t+T)}} \leq C\|(\phi , \vec{h})\|_{X_T} e^{-\gamma
t}\] for any $t\geq 0$.
\end{thm}
The following two technical lemmas are needed for the proof of
Theorem \ref{th2.8}.
\begin{lem}\label{lem2.1}  Let $T>0$ be given. There exists a constant $C=C_T >0$
such that
\begin{itemize}
\item[(i)] for any $u, \ v\in Y_{(0,T)}$,
\[\| (uv)_x\|_{L^1(0,T;L^2(0,L))}\leq C_T \| u\|_{Y_{0,T}}\|v\|_{Y_{(0,T)}}; \]
\item[(ii)] for $f\in L^1 (0,T; L^2 (0,L))$, let
\[ u=\int ^t_0 W(t-\tau) f(\tau ) d\tau ,\]
then
\[ \|u\|_{Y_{(0,T)}} \leq C_T\| f\|_{L^1 (0,T; L^2 (0,L))} .\]
\end{itemize}

\end{lem}
\begin{lem}\label{lem2.2}
Consider a sequence $\{y_n\} ^{\infty}_0$ in a Banach space $X$
generated by iteration as follows:
\begin{equation}
y_{n+1}= Ay_n + F(y_n )\, , \quad n=0,1,2, \cdots . \label{3.1-1}
\end{equation}
Here, the linear operator  $A$ is bounded  from $X$ to $X$ with
\begin{equation} \|A y_n \|_{X} \leq \gamma\|y_n\|_{X} \label{3.2-1}
\end{equation}
for some finite  value $\gamma$ and all $n\geq 0$.  The nonlinear
function $F$ mapping $X$ to $X$ is such that there is constant
$\beta $  and a sequence $\{b_n\}_{n \geq 0}$ for which
\begin{equation}
 \|F(y_n)\|_X \leq \beta  \|y_n\|_X^2 + b_n \label{3.3.1}
\end{equation}
for all $n\geq 0 .$
\begin{itemize} \item[(i)] If   $0<\gamma <1$,
then
 the sequence $\{ y_n\}_{n=0}^\infty $ defined
by (\ref{3.1-1}) satisfies
\begin{equation}
\|y_{n+1} \|_{X} \leq \gamma ^{n+1} \|y_0\|_X +\frac{b^*}{1-\gamma }
\label{3.4-2}
\end{equation}
for any $n \geq 1$, where $b^* = \sup _{n\geq 0} b_n$.
\item[(ii)] If, in addition,
\[ b_{n+1} \leq \delta  ^n c_n  \]
with some finite value of $\delta $ , then
\begin{equation}
\|y_{n+1} \|_{X} \leq \gamma ^{n+1} \|y_0\|_X + nr ^n c^*
\label{3.4-3}
\end{equation}
for any $n \geq 1$, where $r=\max \{ \gamma, \ \delta \}$ and $c^* =
\sup _{n\geq 0} c_n$.
\end{itemize}

\end{lem}

\bigskip
The proof of this Lemma  \ref{lem2.1} can be found in \cite{kz}. As
for Lemma \ref{lem2.2}, its proof is similar to that of Lemma 3.2 in
\cite{bsz-forcing} with just a minor modification.

\bigskip

 {\bf Proof of Theorem \ref{th2.8}:}    Rewrite (\ref{SP}) in its integral form
\[ u(t)=W(t)\phi +W_b (t)\vec{h} -\int ^t_0 W(t-\tau )(au)_x (\tau ) d\tau
.\] Thus, for any $T>0$, using Proposition \ref{decay-1} and
Proposition \ref{pro2.9},
\begin{eqnarray*}
\| u(\cdot, T)\|  &\leq & Ce^{-\beta T}\| \phi \|
+C_T\| \vec{h}\|_{B_{(0,T)}} +\int ^T_0 \|(au)_x \|  (\tau) d\tau \\ \\
&\leq & Ce^{-\beta T}\| \phi \|  +C_T\| \vec{h}\|_{B_{(0,T)}}
+C_T\|a\|_{Y_{(0,T)}}\|u\|_{Y_{(0,T)}} \\ \\ &\leq & Ce^{-\beta T}\|
\phi \|  +C_T\| \vec{h}\|_{B_{(0,T)}} +C_T\|a\|_{Y_{(0,T)}}\mu
(\|a\|_{Y_{(0,T)}})\left (\| \phi \|
+C_T\| \vec{h}\|_{B_{(0,T)}} \right )\\ \\
&\leq & Ce^{-\beta T}\|\phi \|  +C_T\|a\|_{Y_{(0,T)}}\mu
(\|a\|_{Y_{(0,T)}})\|\phi \|  + C_T (1+\|a\|_{Y_{(0,T)}}\mu
(\|a\|_{Y_{(0,T)}})) \| \vec{h}\|_{B_{(0,T)}}.
\end{eqnarray*}
 Note that in the above
estimate, the constant $C$ is independent of $T$. Let $$y_n =
u(\cdot, nT) \ \mbox{ for $n=0, 1, 2, \cdots $}$$ and let $v$ be the
solution of the IBVP

\begin{equation}
\left \{ \begin{array}{l} v_t + v_x +(av)_x +v_{xxx}=0, \qquad
v(x,0) = y_n (x),
\\ \\ v(0,t) = h_1(t+nT), \qquad v_x(L,t)=h_2(t+nT), \qquad v_{xx} (L,t)=h_3(t+nT)
\end{array} \right. \label{2.5-3}
\end{equation}
Thus $y_{n+1} (x)= v(x,T)$ by the semigroup property of the system
(\ref{SP}). Consequently, we have the following estimate for
$y_{n+1}:$ \begin{eqnarray*}
 \| y_{n+1}\| &\leq  & Ce^{-\beta T}\|y_n
\|  +C_T\|a\|_{Y_{(nT,(n+1)T)}}\mu (\|a\|_{Y_{(nT,(n+1)T)}})\|y_n \|  \\
\\ &+ & C_T (1+\|a\|_{Y_{(nT,(n+1)T)}}\mu (\|a\|_{Y_{(nT,(n+1)T)}}))
\|\vec{h} |_{B_{(nT,(n+1)T)}}\end{eqnarray*}
 for $n=0,1,2, \cdots $.
Choose $T$ and $\delta $ such that
\[ Ce^{-\beta T}=\gamma <1,\quad  \gamma + C_T \delta \mu (\delta
):= r < 1.\] Then,
\[ \|y_{n+1}\|  \leq r \| y_n \|   + b_n \]
for all $n\geq 0$ if $\| a\|_{Y_T} \leq \delta $ where
\[ b_n= C_T (1+\|a\|_{Y_{(nT,(n+1)T)}}\mu
(\|a\|_{Y_{(nT,(n+1)T)}})) \|\vec{h}\|_{B_{(nT,(n+1)T)}} .\] It
follows from Lemma \ref{lem2.2} that
\[ \|y_{n+1}\|  \leq r\|y_n\|  +\frac{b^*}{1-r}\]
or
\[ \|y_{n+1}\|  \leq r\|y_n\|  + n\delta ^ nc^* \]
 with $\delta = e^{-\eta T}$ for any $n\geq 0$  depending if (\ref{ccc}) holds
 where $b^*= \sup_{n\geq 0} b_n $ and $c^*= \sup _{n\geq 0} c_n$.
 These
inequalities imply the conclusion of Theorem \ref{th2.8}. $\Box$

\section{Nonlinear problems}
\setcounter{equation}{0}
 In this section we consider the IBVP of the
nonlinear KdV equation posed on the finite domain $(0,L)$:

\begin{equation}\label{3.1}
    \begin{cases}
        u_t+ uu_x + u_x + u_{xxx}=0, \quad  x\in(0,L),\\
        u(x,0)=\phi(x),\\
        u(0,t)=h_1(t), \quad u_x(L,t)=h_2(t), \quad
        u_{xx}(L,t)=h_3(t).
    \end{cases}
\end{equation}
According to Theorem B,\emph{ for given $(\phi, \vec{h})\in X_{(0,
T)}$, there exists a $T^*\in (0,T]$   such that the IBVP (\ref{3.1})
admits a unique solution $u\in Y_{(0,T^*)}$.} This well-posedness
result is temporally local in the sense that  the  solution $u$ is
only guaranteed to exist on the time interval $(0, T^*)$, where
$T^*$ depends on the norm of $(\phi ,\vec{h})$ in the space
$X_{(0,T)}$. The next proposition presents an alternative view of
local well-posedness for the IBVP (\ref{3.1}). If the norm of
$(\phi, \vec{h})$ in $X_{(0,T)}$ is not too large, then the
corresponding solution is guaranteed to exist over the entire time
interval $(0, T)$.
\begin{prop}\label{reg}
Let $T>0$ be given.  There exists $\delta >0$ such that if $\|(\phi,
\vec{h})\|_{X_{(0,T)}} \le \delta $, then the solution $u$  of the
IBVP (\ref{3.1}) belongs to the space $Y_{(0,T)}$ and, moveover,
there exists a constat $C>0$ depending only on $T$  and $\delta $
such that
\[ \| u\|_{Y_{((0,T))}} \leq C\| (\phi , \vec{h})\|_{X_{(0,T)}} .\]
\end{prop}

\begin{proof}  For $(\phi ,
\vec{h})\in X_{(0,T)}$, define a map $\Gamma : Y_{(0,T)} \to
Y_{(0,T)}$ by
\[ \Gamma (v)= W(t) \phi(x)+ W_b (t) \vec{h} -\int_0^t W(t-\tau) \left (vv_x
\right) (\tau)d\tau.\] By Lemma \ref{lem2.1} and Proposition
\ref{pro2.9}, for any $v, \, v_1, \, v_2 \in Y_{(0,T)}$,
\[ \| \Gamma (v)\|_{Y_{(0,T)}}\leq C_1 \| (\phi , \vec{h})\|_{X_{(0,T)}}
+C_2\| v\|^2 _{Y_{(0,T)}} \] and
\[ \| \Gamma (v_1)-\Gamma (v_2)\|_{Y_{(0,T)}}\leq C_2
\|(v_1+v_2\|_{Y_{(0,T)}}\|v_1-v_2\|_{Y_{(0,T)}} \] where $C_1 $ and
$C_2$ are constants depending only on $T$. Choose
\[ \delta = \frac{1}{8C_1C_2}, \qquad M=2C_1 \delta .\]
Let
\[ S_M = \{ v\in Y_{(0,T)}, \ \|v\|_{Y_{(0,T)}} \leq M \} .\]
Assume that $\|(\phi , \vec{h})\|_{X_{(0,T)}}\leq \delta$. Then, for
any $v, \, v_1, \, v_2 \in M$,
\[  \| \Gamma (v)\|_{Y_{(0,T)}}\leq C_1 \| (\phi , \vec{v}\|_{X_{(0,T)}}
+C_2\| v\|^2 _{Y_{(0,T)}} \leq C_1 \delta +4C_2 C_1^2 \delta ^2\leq
\frac32C_1 \delta \leq M \] and
\[  \| \Gamma (v_1)-\Gamma (v_2)\|_{Y_{(0,T)}}\leq 2C_2 M
 \|v_1-v_2\|_{Y_{(0,T)}}\leq \frac12 \|v_1-v_2\|_{Y_{(0,T)}}. \]
 Thus $\Gamma $ is a contraction in the ball $S_M$ and its fixed
 point $u\in Y_{(0,T)}$ is the desired solution of the the IBVP
 (\ref{3.1}) which, moreover, satisfies
 \[ \|u\|_{Y_{(0,T)}}\leq \frac34 C_1 \|(\phi , \vec{h})\|_{X_{(0,T)}}
 .\]
 \end{proof}

Next we show that if the initial value $\phi$ and boundary value
$\vec{h}$ are small, then the corresponding solution $u$ exists for
any time $t>0$ and, moreover,  its norm in the space $L^2 (0,L)$ is
uniformly bounded.
\begin{prop}\label{3-1}
There  exist positive constants $T$, $\delta _j,\ j=1,2 $ and $r$
such that if $(\phi, \vec{h})\in X_T$ satisfying
\begin{equation} \| \phi\|  \leq \delta _1,\quad \|\vec{h} \|_ {B_T}\leq \delta_2  \label{c-1}
\end{equation}
then the corresponding  solution $u$ of (\ref{3.1}) is globally
defined and belongs to the space $Y_T$. Moreover,
\[ \| u(\cdot, t)\|  \leq C_1 e^{-rt}\| \phi \|  + C_2 \|\vec{h}\|_{B_T} \]
for any $t\geq 0$ and \[ \| u\|_{Y_T} \leq C_3 \| (\phi , \vec{h}\|
_{X_T},\] where $C_1$, $C_2 $  and $C_3$ are constants depending
only on $T$, $\delta _1 $ and $\delta _2$.
\end{prop}
\begin{proof}
 For given $\phi \in L^2 (0,L)$ and $\vec{h}\in B_T$, rewrite the IBVP  (\ref{3.1})
 in its integral form
\[ u(t)=W(t) \phi +W_b (t) \vec{h} -\int ^t_0 W(t-\tau )(uu_x)(\tau ) d\tau  .\]
For any given $T>0$,
 there exist $c_1
>0$ independent of $T$ and $c_2, \ c_3$ depending only
 on $T$ such that for any $0\leq t\leq T$,
 \begin{equation}
\| u (\cdot, t)\|   \leq c_1 e^{-rt} \| \phi \|_{L^2(0,L)} + c_2 \|
u\|_{Y_{(0,T)}}^2 + c_3 \|\vec{h}\|_{B_{0,T}}. \label{x-1}
\end{equation}
By Proposition \ref{reg}, there exists a $\delta  >0$ and a constant
$c_4 >0$ such that  if
\begin{equation}
\| (\phi , h)\|_{X_{(0,T)}}\leq \delta, \label{x-2}
\end{equation}
then
\[ \| u\|_{Y_{(0,T)}} \leq c_4 \|(\phi , h)\|_{X_{(0,T)}} .\]
Thus, if (\ref{x-2}) holds and (\ref{x-1}) is evaluated at $t=T$,
\[ \| u(\cdot, T)\|  \leq c_1 e^{-rT} \| \phi \| + c_5 (\| \phi \| ^2
+ \| \vec{h} \|_{B_{(0,T)}}^2 ) + c_3 \|\vec{h} \|_{B_{(0,T)}}\]
with $c_5= c_2c_4^2$. Since $c_1 $ is independent of $T$, one can
choose $T>0$ so that $ c_1e^{-rT} =\gamma <\frac12$.
 Then choose $\delta _1 $
and $\delta _2$ such that
\[  \delta _1 +\delta _2 \leq \delta \]
and
\[ \frac12 \delta _1 +c_5 (\delta _1^2 +\delta _2 ^2) +c_3 \delta _2   \leq \delta _1 .\]
For such values of $\delta _1 $ and $\delta _2$, we have that
\[ \| u(\cdot, T)\|  \leq \delta _1,\]
and, in addition, by the assumption,
\[ \|\vec{h}\|_{B_{(T,2T)} } \leq \delta _2 .\]
Hence repeating the argument, we have that
\[  \sup _{T\leq t\leq 2T} \| u(\cdot, t)\|  \leq \delta _1, \qquad \|u\|_{Y_{(T, 2T)}}\leq c_4\delta.\]
Continuing inductively, it is adduced that
\[ \sup _{t\geq 0} \| u(\cdot, t)\|  \leq \delta _1, \|u\|_{Y_{T}}
\leq c_4 \delta .\] Let $y_n=u(\cdot, nT)$ for $n=1,2, \cdots$.
Using the semigroup property of (\ref{3.1}),
one obtains
\[ \|y_{n+1}\|  \leq \frac12\|y_n\|  +c_2 \|y_n \|^2 +
c_3 \|h \|_{B_{(nT,(n+1)T)}}\] for any $n\geq 0$ provide $\| y_0\|
\leq \delta _1$ and
\[ \sup _{n\geq 0}\|\vec{h} \|_{B_{(nT,(n+1)T)}} \leq \delta _2 .\]
By Lemma \ref{lem2.2}, there exists $0<\nu <1$, $\delta _1^*>0$ and
$\delta _2^*>0$ such that if
\[ \|y_0 \|  \leq \delta _1^*, \quad b_n = C_3\|\vec{h} \|_{B_{(nT,(n+1)T)}}
\leq \delta _2 \leq \delta _2^*\] for all $n\geq 0$, then
\[ \| y_{n+1}\|  \leq \nu ^{n+1} \|y_0 \| + \frac{b^*}{1-\nu} \]
for all $n\geq 1$, where $b^* = \max _n \{b_n\}.$ This leads by
standard arguments to the conclusion of Proposition \ref{3-1}.
\end{proof}

\begin{prop}\label{3-2} Under the assumptions of Proposition
\ref{3-1}, if
\[  \|\vec{h}\|_{B_{(t, t+T)}}\leq g(t) e^{-\nu t} \ for \ all \
t\geq 0 \] with $\nu >0$ and $g\in B_T$ and $\| g\|_{B_T} \leq
\delta _2$,  then there exist a $0< \gamma \leq \max \{ r, \nu \} $
and $C>0$  such that
\[  \|u\|_{Y_{(t, t+T)}} \leq C\|(\phi , \vec{h})\|_{X_T} e^{-\gamma
t}\] for any $t\geq 0$.
\end{prop}
\begin{proof}
Setting $a=\frac12 u$, the equation in (\ref{3.1}) becomes
\[ u_t +u_x +(au)_x +u_{xxx}=0. \]
Proposition \ref{3-2} follows from Proposition \ref{3-1} and Theorem
\ref{th2.8} as a corollary.
\end{proof}

\medskip
Now we are at the stage to present the proofs of Theorem
\ref{global} and Theorem \ref{long time}.

\bigskip
{\bf Proof of Theorem \ref{global}}:   We only consider the case of
$0\leq s\leq 3$. The proof for the case of $s>3$ is similar. Note
first the theorem is Proposition \ref{3-1} when $s=0$. For $s=3$,
let $v=u_t $. Then $v$ solves
  \begin{equation}\label{TP}
        \begin{cases}
        v_t+v_x+v_{xxx}+(uv)_x=0,\qquad    x\in (0,L),  \ t >0,
        \\ v(x,0)=\phi ^*(x), \\
        v(0,t)=h'_1(t),\ u_x(L,x)=h'_2(t),\ v_{xx}(L,t)=h'_3(t)
        \end{cases}
    \end{equation}
    with $\phi ^*(x)= -\phi '(x)-\phi (x)\phi'(x)-\phi '''(x) $.
    Note that
    \[ \| u\|_{Y_T} \leq C\| (\phi. \vec{h}\|_{X_T} .\]
    Thus, invoking Theorem \ref{th2.8} yields  that $v\in Y_T$ and
    \[ \|v\|_{Y_T} \leq C \|(\phi ^*, \vec{h}')\|_{X_T} .\]
    Then,  it follows from the equation
    \[ u_{xxx}= -u_t -uu_x -u_x = -v -uu_x -u_x \]
    that $u\in Y^3_T$ and
    \[ \| u\| _{Y^3_T} \leq C \|(\phi , \vec{h}\|_{X^3_T} \]
    for some constat $C>0$ Thus, Theorem \ref{global} holds for
    $s=3$. The case of $0< s< 3$ then follows by using the nonlinear
    interpolation theory of Tartar \cite{tartar, bsz-half-1}. $\Box
    $

    \bigskip
    {\bf Proof of Theorem \ref{long time}}: For any $r>0$,
    define the space \[X^s _T (r)= \{ (\phi, \vec{h})\in X^s_T|\  e^{rt}(\phi , \vec{h}) \in
    X_T^s\} \]
    and
    \[ Y^s_T (r)= \{ u\in Y^s_T| \ e^{rt}u\in Y^s_T\}.\]
    Equipped with the norms
    \[ \| (\phi , \vec{h})\|_{X^s_T (r)}:= \| (\phi ,
    e^{rt}\vec{h})\|_{X^s_T} \]
    and
    \[ \| u\|_{Y^s_T(r)}: =\| e^{rt} u\|_{Y^s_T}, \]
    respectively, both $X^s_T(r)$  and $Y^s_T (r)$ are Banach
    spaces.  Theorem  \ref{long time} can then  be restated as

    \medskip
    \emph{For given $s\geq 0$ and $\nu >0$ with
\[ s\ne \frac{2j-1}{2}, \quad j=1,2,3,\cdots ,\] there exist  positive constants $T$, $\gamma
    $, $\delta $ and $C$ such that for $s-$compatible $(\phi ,
    \vec{h})\in X^s_T (\nu) $ with
    \[ \|(\phi , \vec{h})\|_{X^s_T (\nu)} \leq \delta , \]
    the corresponding solution $u$ of the IBVP(\ref{3-1}) belongs to
    the space $Y^s_T (\gamma)$ and
    \[ \| u\| _{Y^s_T(\gamma)} \leq C \| (\phi , \vec{h})\|_{X^s_T
    (\nu)} .\]}

    \medskip
    Its proof is  similar to    that of Theorem \ref{global}. $\Box$

    \section{Concluding remarks}
    \setcounter{equation}{0}
    The focus of our discussion has been the well-posedness of the initial value problem
    of the KdV equation posed on the finite interval $(0,L)$:

    \begin{equation}\label{4.1}
        \begin{cases}
        u_t+u_x+u_{xxx}+uu_x=0,\qquad    x\in (0,L),  \ t >0,
        \\ u(x,0)=\phi(x), \\
        u(0,t)=h_1(t),\ u_x(L,x)=h_2(t),\ u_{xx}(L,t)=h_3(t).
        \end{cases}
    \end{equation}
    It is considered with the initial data $\phi \in H^s (0,L)$ and
    the boundary data $\vec{h}=(h_1, h_2 , h_3 )$ belongs to the
    space $B^s_{(0,T)}= H^{\frac{s+1}{3}}(0,T)\times H^{\frac{s}{3}}
    (0,T)\times H^{\frac{s-1}{3} } (0, T)$ with $s\geq 0$.  Although
    the IBVP is known to be locally (in time) well-posed, whether
    solutions exist globally is still an open question because even
    the simplest global \emph{a priori $L^2-$ }estimate is not
    available for solutions of the IBVP (\ref{4.1}). As a partial
    answer to this open problem of the global well-posedness, we
    have shown in this article that the solutions exist  globally (in time) in
    the space $H^s(0,L)$  for any $s\geq 0$ as long as its  auxiliary data
    $(\phi , \vec {h})$ is small in the space $X^s_T$, which is an
    improvement of Colin and Ghidalia's early work \cite{col-ghida}.
    In addition, we have studied the long time behavior of those
    globally existed solutions and have shown  that those small amplitude
    solutions decay exponentially if their boundary value
    $\vec{h}(t)$ decays exponentially. In particular, those
    solutions satisfying homogenous boundary conditions decay
    exponentially in the space $H^s (0,L)$ if their initial values
    are small in  $H^s (0,L)$.

    \medskip

It is interesting to compare the IBVP (\ref{4.1}) with another
well-studied IBVP of the KdV equation posed on the finite interval
$(0,L)$:
    \begin{equation}\label{4.2}
        \begin{cases}
        u_t+u_x+u_{xxx}+uu_x=0,\qquad   x\in (0,L),  \ t >0,
        \\ u(x,0)=\phi(x), \\
        u(0,t)=h_1(t),\ u (L,x)=h_2(t),\ u_{x}(L,t)=h_3(t).
        \end{cases}
    \end{equation}
While the well-posedness results as described by Theorem C, Theorem
1.1 and Theorem 1.2 for the IBVP (\ref{4.1}) are all true for the
IBVP (\ref{4.2}), the IBVP (\ref{4.2}) is globally well-posed in the
space $H^s (0,L)$ for any $s\geq 0$ \cite{bsz-finite-1,Fam07}:

\medskip
\emph{for given $s-$compatible $(\phi ,\vec{h})\in H^s(0,L)\times
H^{\frac{s^*+1}{3}} (0,T)\times H^{\frac{s^*+1}{3}} (0,T)\times
H^{\frac{s^*}{3}} (0,T)$, the IBVP (\ref{4.2}) admits a unique
solution $u\in Y^s_{(0,T)}$. Here $ s^*=s^+ $ if $0\leq s< 3$ and
$s^*=s $ if $s\geq 3$.}

\medskip
The reason for the IBVP (\ref{4.2}) to be globally well-posedness is
simply that global \emph{a priori} $L^2$ estimate holds for
solution $u$ of the IBVP (\ref{4.2}) with homogeneous boundary
conditions;
\[ \frac{d}{dt}\int ^L_0 u^2(x,t)dx +u^2_x(0,t) =0 \quad for \ any \
t\geq 0.\]

Thus whether solutions of the IBVP (\ref{4.1}) exist globally or
blow up in finite time becomes really interesting. If it does not
blow up in finite time, then how to establish its global
well-posedness without knowing if its simplest global $L^2$ \emph{a
priori} estimate holds or not? As Archimedes said, ``\emph{Give me a
place to stand and with a lever I will move the whole world}.'' For
the IBVP (\ref{4.1}), if there are no global \emph{a priori}
estimates available,  how to prove its global well-posedness? On the
other hand, if some solutions of the IBVP (\ref{4.1}) do blow up in
finite time, that would be also very interesting since the blow up
would be mainly caused by the boundary conditions rather than the
nonlinearity of the equation. We are not aware of any such kind of
results existed in the literature.

Finally we  point out that, started by the work of Ghidalia
\cite{ghida-1} in 1988, the KdV equation posed on a finite domain
has also been studied intensively from dynamics point of views
\cite{bsz-forcing, ghida-1,ghida-2,12,uz-1, uz-2,ybz,bz-1, bz-2}.
One of the questions people are interested is whether the equation
admits a time periodic  solution if the external forcing functions
are time periodic. Such a time periodic solution, if exists, is
called \emph{forced oscillation}, which can be viewed as a
\emph{limit cycle} from dynamics point of view. A further question
to study for this \emph{limit cycle} is: what is its stability? In
\cite{uz-2}, Usman and Zhang has obtained the following result for
the following system  associate to the IBVP (\ref{4.2}):
   \begin{equation}\label{4.3}
        \begin{cases}
        u_t+u_x+u_{xxx}+uu_x=0,\qquad   x\in (0,L),  \ t >0,
        \\
        u(0,t)=h_1(t),\ u (L,x)=0,\ u_{x}(L,t)=0.
        \end{cases}
    \end{equation}

\medskip
{\bf Theorem D:} \emph{There exists a $\delta >0$ such that if
$h_1\in H^{\frac13}_{loc} (R^+)$ is a time -periodic function of
period $\tau $ satisfying $\| h_1\|_{H^{\frac13} (0, \tau )} \leq
\delta $, then (\ref{4.3})  admits a time periodic solution
$$u^* \in C_b (0, \infty ; L^2 (0,L))\cap L^2_{loc} (0, \infty; H^1
(0,L)),$$ which is locally exponentially stable.}

\medskip  The same result holds for following system associated to the IBVP (\ref{4.1}):
   \begin{equation}\label{4.4}
        \begin{cases}
        u_t+u_x+u_{xxx}+uu_x=0,\qquad   x\in (0,L),  \ t >0,
        \\
        u(0,t)=h_1(t),\ u _x(L,x)=h_2(t),\ u_{xx}(L,t)=h_3(t).
        \end{cases}
    \end{equation}
    In fact, using the
same approach as  that in \cite{uz-2} with a slight modification, we
have the following theorem for the system (\ref{4.4}).
\begin{thm}  There exists a $T>0$ and
$\delta >0$ such that if $\vec{h}\in B_T$ is  a time -periodic
function of period $\tau $ satisfying
\[ \| \vec{h}\|_{B_T} \leq \delta ,\]
then system  (\ref{4.4}) admits a admits a time periodic solution
$$u^* \in  Y_T,$$ which is locally exponentially stable.
\end{thm}

\bigskip
\textbf{Acknowledgments}. Ivonne Rivas was partially supported by
the Taft Memorial Fund at the University of Cincinnati through
Graduate Dissertation Fellowship. Bing-Yu Zhang was partially
supported  by the Taft Memorial Fund at the University of Cincinnati
and  the Chunhui program (State Education Ministry of China) under
grant Z007-1-61006.



\begin{thebibliography}{99}


\bibitem{bsz-half-1} J. L. Bona, S. Sun and B.-Y. Zhang, A
nonhomogeneous boundary-value problem for the Korteweg-de Vries
equation in a quarter plane, {\em Trans. American Math. Soc.}\, {\bf
354}\,(2001), 427 -- 490.

\bibitem{bsz-forcing}   J. L. Bona, S. M. Sun and B.-Y. Zhang, Forced oscillations
of a damped Korteweg-de Vries equation in a quarter plane, {\it
Comm. Contemp. Math.}, 2003, {\bf 5}(3), 369--400.

\bibitem{bsz-finite-1}  J. L. Bona, S. Sun and B.-Y. Zhang, A nonhomogeneous
boundary value problem for the KdV equation posed on a finite
domain,  {\em Commun. Partial Differential Eq.}\,{\bf 28}\,(2003),
1391--1436.

\bibitem{bsz-half-2}  J. L. Bona, S. Sun and B.-Y. Zhang, Non-homogeneous
 Boundary Value Problems for  the Korteweg-de Vries  and the Korteweg-de Vries-Burgers
 Equations in a Quarter Plane (with J. L. Bona and S. M. Sun),
 {\em  Annales de l'Institut Henri Poincar\'e - Analyse non
 lin\'eaire} {\bf 25}\,(2008), 1145-1185.

\bibitem{bsz-finite-2}  J. L. Bona, S. Sun and B.-Y. Zhang, The Korteweg-de Vries equation on a finite
domain II  (with J. L.
Bona and S. M. Sun),   \emph{J. Diff. Eqns} {\bf 247}\,(2009),
2558--2596.


\bibitem{bub-1} B. A. Bubnov,  Generalized boundary value problems for the
Korteweg-de Vries equation in bounded domain, \emph{Differential
Equations} \textbf{15}(1979), 17--21.

\bibitem{bub-2} B. A. Bubnov,  Solvability in the large of nonlinear
boundary-value problem for the Korteweg-de Vries equations,
\emph{Differential Equations, }\textbf{16}(1980), 24--30.

\bibitem{ck} J. Colliander and C. Kenig,  The
generalized Korteweg--de Vries equation on the half line,
 \emph{Comm. Partial Differential Equations }, \textbf{27}(2002), 2187--2266.

\bibitem{col-ghida} T. Colin and J.-M.  Ghidaglia, An initial-boundary-value problem fo
the Korteweg-de Vries Equation posed on a finite
interval,\emph{ Adv. Differential Equations }\textbf{6} (2001),
1463-–1492.

\bibitem{6}  P. Constantin, C. Foias, B.  Nicolaenko and R. Temam,
Integral manifolds and inertial manifolds for dissipative partial
differential equations,  {\em  Applied Mathematical Sciences,} Vols.
70,  Springer-Verlag, New York-Berlin, 1989.

\bibitem{7}   W. Craig and  C. E. Wayne,  Newton's method and periodic
solutions of nonlinear wave equations, {\em  Comm. Pure Appl.
Math.}, {\bf 46}\,(1993), 1409--1498.

\bibitem{Fam83} A. V. Faminskii, A. V.,  The Cauchy problem and the
mixed problem in the half strip for equation of Korteweg-de Vries
type, (Russian) \emph{Dinamika Sploshn. Sredy }\textbf{162} (1983),
152--158.

\bibitem{Fam89} A. V. Faminskii, A mixed problem in a semistrip for
the Korteweg-de Vries equation and its generalizations,
(Russian)\emph{ Dinamika Sploshn. Sredy }\textbf{258} (1988),
54--94; English transl. in \emph{Trans. Moscow Math. Soc.
}\textbf{51} (1989), 53--91

\bibitem{Fam99} A. V. Faminskii,  Mixed problms fo the Korteweg-de Vries equation,
\emph{Sbornik: Mathematics} \textbf{190} (1999), 903--935.

\bibitem{Fam01} A. V. Faminskii,  On an initial boundary value problem in
a bounded domain for the generalized Korteweg-de Vries equation,
\emph{International
 Conference on Differential and Functional
 Differential Equations }(Moscow, 1999),  \emph{Funct. Differ. Equ.} \textbf{8} (2001),  183--194.



\bibitem{Fam04} A. V. Faminskii, An initial boundary-value problem in a half-strip for the
Korteweg-de Vries equation in fractional order Sobelev Spaces,\emph{
Comm. Partial Differential Eq.} \textbf{29} (2004), 1653--1695.

\bibitem{Fam07}  A. V. Faminskii,  Global well-posedness of two initial-boundary-value
problems for the Korteweg-de Vries equation, \emph{Differential
Integral Equations },\textbf{20} (2007),
   601--642.

\bibitem{ghida-1}   J.-M. Ghidaglia,  Weakly damped forced Korteweg-de Vries
equations behave as a finite-dimensional dynamical system in the
long time, {\it  J. Differential Eqns.} \textbf{74}\,(1988),
369--390.

\bibitem{ghida-2}  J.-M. Ghidaglia,   A note on the strong convergence towards
attractors of damped forced KdV equations, {\it J. Differential
Eqns.} \textbf{110}\,(1994),   356--359.

\bibitem{holmer} J. Holmer, The initial-boundary value problem for the
Korteweg-de Vries equation,  \emph{Comm. Partial Differential
Equations} {\bf 31} (2006), 1151--1190.

\bibitem{kato} T. Kato, \emph{Perturbation Theory for  Linear
Operators,} Dir Grundlehren der mathematischen Wissenschaften, {\bf
132}, Springer, New York, 1966.

\bibitem{10}  J. B. Keller and L. Ting, Periodic vibrations of
 systems governed by nonlinear partial differential equations,
{\em Comm. Pure and  Appl. Math.} 1966, {\bf 19}(2),  371 -- 420.

\bibitem{11}    J. U. Kim, Forced vibration of an aero-elastic plate, {\em
J. Math. Anal. Appl.} {\bf 113}\,(1986),  454--467.

\bibitem{kz} G. Kramer and B.-Y. Zhang,  Nonhomogeneous boundary value problems
of the KdV equation on a bounded domain,  \emph{Journal Syst. Sci.
\& Complexity}, {\bf 23}\,(2010), 499-526.


\bibitem{Pazy} A. Pazy,  \emph{Semigroups of linear operators and applications to
partial differential equations}, Applied Mathematical Sciences, {\bf
44}, Springer-Verlag, New York, 1983.

\bibitem{rkz} I. Rivas, G. Kramer and B.-Y. Zhang, Well-posedness of
a class of initial-boundary-value problem for the Kortweg-de Vries
equation on a bounded domain, preprint.


\bibitem{13}  P. H. Rabinowitz, Periodic solutions of nonlinear
hyperbolic differential equations, {\em Comm. Pure Appl. Math.} {\bf
20}\,(1967), 145 -- 205.


\bibitem{14}   P. H. Rabinowitz, Free vibrations  for a semi-linear wave
equation, {\em Comm. Pure Appl. Math.} {\bf 31}\,(1978), 31 -- 68.


\bibitem{12}  G. R. Sell and Y. C.  You,  Inertial manifolds: the
nonselfadjoint case, {\it J. Diff. Eqns.}  \textbf{96}\,(1992),
203--255.

\bibitem{tartar} L. Tartar, Interpolation non lin\'eaire et r\'egularit\'e,
\emph{J. Funct. Anal} \textbf{9} (1972), 469--489.

\bibitem{15}    O. Vejvoda, {\em partial differential equations:
time-periodic solutions}, Mrtinus Nijhoff Publishers, 1981.

\bibitem{16}   C. E. Wayne,  Periodic solutions of nonlinear partial
differential equations, {\em  Notices Amer. Math. Soc.} {\bf  44
}\,(1997), 895--902.

\bibitem{uz-2} M. Usman and B.-Y. Zhang, Forced oscillations of the Korteweg-de Vries equation
 on a bounded domain and their stability, {\em J. Systems Sciences and Complexity} {\bf 20}\,(2007),
15--24.

\bibitem{uz-1} M. Usman, B.-Y. Zhang, Forced oscillations of a class of nonlinear dispersive wave
equations  and their stability,  \emph{Discrete and Continuous
Dynamical Systems}, \textbf{26} (2010), no. 4, 1509--1523.


\bibitem{ybz}   Y. Yang and B.-Y. Zhang, Forced oscillations of a damped
Benjamin-Bona-Mahony equation in a quarter plane, {\it  Control
theory of partial differential equations,}   Lect. Notes Pure Appl.
Math., Vols. 242, Chapman \& Hall/CRC, Boca Raton, FL, 2005,
375--386.

\bibitem{bz-1}   B.-Y. Zhang,  Forced oscillations of a regularized
long-wave equation and their global stability, {\it Differential
equations and computational simulations} (Chengdu, 1999), World Sci.
Publishing, River Edge, NJ, 2000, 456--463.

\bibitem{bz-2}  B. -Y. Zhang,  Forced oscillation of the Korteweg-de
Vries-Burgers equation and its stability, {\it Control of nonlinear
distributed parameter systems} (College Station, TX, 1999),
 Lecture Notes in Pure and Appl. Math., Vols. 218, Dekker,
New York, 2001, 337-357.


\end{thebibliography}

\end{document}